\documentclass[reqno]{amsart}
\usepackage{a4wide,calrsfs}
\usepackage{amsmath,bbm}
\usepackage{color}
\usepackage{epsfig}
\usepackage{subfigure}
\usepackage{amssymb}
\usepackage{enumerate}
\usepackage{graphicx}
\usepackage{subfigure}
\usepackage{amsmath}
\usepackage{rotating}
\usepackage{stmaryrd}
\usepackage{upgreek}

\newcommand{\R}{\mathbb{R}}

\newcommand{\re}{\mathrm{Re}}

\newtheorem{theorem}{Theorem}[section]
\newtheorem{proposition}{Proposition}[section]

\newtheorem{lemma}{Lemma}[section]

\allowdisplaybreaks
\numberwithin{equation}{section}

\title[Asymptotic Integration of a Linear Fourth Order ODE]
{Asymptotic Integration of a Linear Fourth Order Differential Equation
of Poincar\'e Type}

\author[Coronel]{An{\'\i}bal\ Coronel$^\dag$}
\author[Huancas]{Fernando Huancas$^\dag$}
\author[Pinto]{Manuel Pinto$^\ddag$}

\thanks{$^\dag$ GMA, Departamento de Ciencias B\'asicas,
Facultad de Ciencias, Universidad del B\'{\i}o-B\'{\i}o,
Campus Fernando May, Chill\'{a}n, Chile,
E-mail: {\tt acoronel@ubiobio.cl, fihuanca@gmail.com}}

\thanks{$^\ddag$
	      Departamento de Matem\'atica,
	      Facultad de Ciencias, Universidad de Chile,
	      Santiago, Chile,
	      E-mail: {pintoj@uchile.cl}}

\date{\today}

\begin{document}

\begin{abstract}
This article deals with the asymptotic behavior of fourth order 
differential equation where the coefficients are perturbations 
of linear constant coefficient equation. We introduce a change 
of variable and deduce that the new variable
satisfies a third order differential equation of 
Riccati type. We assume three hypothesis. The first
is the following: all roots of the characteristic 
polynomial associated to the fourth order linear equation has distinct
real part. The other two hypothesis 
are related with the behavior of 
the perturbation functions.  Under this general 
hypothesis we obtain four main results. The first two results
are related with the application 
of fixed point theorem to prove that the Riccati equation has 
a unique solution. The next result
concerns with the 
asymptotic behavior
of the solutions of the Riccati equation. 
The fourth main theorem is introduced to establish the existence 
of a fundamental system of solutions and to precise 
formulas for the asymptotic behavior of 
the linear fourth order  differential equation.
\end{abstract}


\maketitle

\section{Introduction}

In this paper we are interested in the following fourth order differential equation
\begin{eqnarray}
y^{(4)}+\displaystyle\sum_{i=0}^3[a_i+r_i (t)] y^{(i)}=0,
\quad
a_i\in \R 
\quad
\mbox{and}
\quad r_i: \R\to \R.
\label{eq:intro_uno}
\end{eqnarray}
This equation is a perturbation of the following constant coefficient equation:
\begin{eqnarray}
y^{(4)}+\displaystyle\sum_{i=0}^3a_i y^{(i)}=0.
 \label{eq:intro_dos}
\end{eqnarray}
The classical analysis of \eqref{eq:intro_uno} is mainly focus on two questions:
the existence of a fundamental system of solutions and the characterization
of the asymptotic behavior of its solutions. The first significative answers of both 
problems  
comes back to the seminal work of Poincar\'e \cite{poincare_1885} and 
has been investigated  by several authors with
long and rich history of results \cite{bellman_book,eastham_book,harris_lutz_1977,hartman_1948}.
However, although it is an old problem    
is a matter which does not lose its topicality and  importance 
in the research community. For instance in the case
of asymptotic behavior  there are the following newer results  
\cite{figueroa_2008,figueroa_2010,barbara_2013,stepin_2005,stepin_2010}.
In particular, in this contribution, we address the question 
of new explicit formulas for asymptotic behavior of nonoscilatory 
solutions for \eqref{eq:intro_uno}  by application of the scalar method 
introduced by Bellman  in \cite{bellman_1949} (see also
\cite{bellman_1950,bellman_book}) and recently 
applied by Figueroa and Pinto \cite{figueroa_2008,figueroa_2010},
Stepin \cite{stepin_2005,stepin_2010} and 
Pietruczuk \cite{barbara_2013}.

Linear fourth-order differential equations appear in several 
areas of sciences and engineering as the more basic mathematical models. 
These simplified equations,  arise from different 
linearization approaches used to give an ideal description
of the physical phenomenon or in order to
analyze (analytically solve or numerically simulate) the 
corresponding  nonlinear governing equations. For instance,
the one-dimensional of Euler-Bernoulli model
in linear theory of elasticity  \cite{aftabizadeh_1986,timoshenko_book},
the optimization of quadratic functionals in optimization
theory \cite{aftabizadeh_1986},
the mathematical model in viscoelastic flows \cite{davies_1988,lamnii_2014},
and the biharmonic equations
in radial coordinates in harmonic analysis \cite{gazzola_2006,karageorgis_2011}.
In particular, here we describe the last application.
We recall that the biharmonic equation
\begin{eqnarray*}
 \Delta^2u(\mathbf{x})=0
 \quad\mbox{in }
 \quad\mathbb{R}^n,
 \quad\mbox{with }
 n\ge 5, 
\end{eqnarray*}
in radial coordinates with $r=\|\mathbf{x}\|$ and $\phi(r)=u(\mathbf{x})$,
may be rewritten as follows
\begin{eqnarray*}
 \phi^{(4)}(r)
 +\frac{2(n-1)}{r}\phi^{(3)}(r)
 +\frac{(n-1)(n-3)}{r^2}\phi^{(2)}(r)
 -\frac{(n-1)(n-3)}{r^3}\phi^{(1)}(r)
 =0,
 \quad
 r\in ]0,\infty[.
\end{eqnarray*}
Now, by introducing the  change of variable $v(t)=e^{-4t/(p-1)}\phi(e^{t})$
for some $p>(n+4)(n-4)^{-1}$,
the differential equation for $\phi$ can be transformed in the following 
equivalent equation
\begin{eqnarray}
 v^{(4)}(t)
 +K_3v^{(3)}(t)
 +K_2v^{(2)}(t)
 +K_1v^{(1)}(t)
 +K_0v(t)
 =0,
 \quad
 t\in  \mathbb{R},
 \label{eq:biarm_radial}
\end{eqnarray}
where 
\begin{eqnarray*}
K_0&=&
\frac{8}{(p-1)^4}\Bigg[
(n-2)(n-4)(p-1)^3
+2(n^2-10n+20)(p-1)^2
-16(n-4)(p-1)+32
\Bigg],
\\
K_1&=&-
\frac{8}{(p-1)^3}\Bigg[
(n-2)(n-4)(p-1)^3
+4(n^2-10n+20)(p-1)^2
-48(n-4)(p-1)+128
\Bigg],
\\
\\
K_2&=&
\frac{1}{(p-1)^2}\Bigg[
(n^2-10n+20)(p-1)^2
-24(n-4)(p-1)+96\Bigg],
\\
K_3&=&
\frac{2}{p-1}\Bigg[
(n-4)(p-1)-8\Bigg],
\end{eqnarray*}
see \cite{gazzola_2006} for further details.
We note that the roots of the characteristic polynomial
associated to the homogeneous equation
are given by
\begin{eqnarray}
 \lambda_1=2\frac{p+1}{p-1}>\lambda_2=\frac{4}{p-1}
 >0>
 \lambda_3=\frac{4p}{p-1}-n>\lambda_4=2\frac{p+1}{p-1}-n.
 \label{eq:insp_h1}
\end{eqnarray}
Thus,  the radial solutions of the biharmonic equation equation in
an space of dimension $n\ge 5$ and with $p>(n+4)(n-4)^{-1}$ can
be analysed by the linear fourth order differential equation
\eqref{eq:biarm_radial} where the  characteristic 
roots satisfy  \eqref{eq:insp_h1} which will be generalized by 
considering throughout of the paper the  assumption (H$_1$).
See the list of assumptions given below at the end of the
introduction.

Nowadays, there exist three big approaches to study the
problem of asymptotic behavior of solutions for
\eqref{eq:intro_uno}: the analytic theory, the nonanalytic 
theory and the scalar method. In a broad sense, we recall 
that the  essence of the analytic theory consist in
the assumption of some  representation of the
coefficients and of the solution,
for instance power series representation (see \cite{codilevi_book} for details). 
Concerning to the nonanalytic theory,
we know that the methods are procedures consisting of the two main 
steps: first a change 
of variable to transform \eqref{eq:intro_uno} in a system 
of first order of Poincare type and then by 
the application of a diagonalization process
to obtain the asymptotic formulas 
(for further details consult 
\cite{eastham_book,coppel_book,fedoryuk_book,olver_book}). 
Meanwhile, in the scalar method  \cite{bellman_book,
figueroa_2008,figueroa_2010,stepin_2005,stepin_2010,barbara_2013,bellman_1949,bellman_1950}
the asymptotic behavior of 
solutions for \eqref{eq:intro_uno} is  obtained by
a change of variable which reduce  \eqref{eq:intro_uno}
to a third order Riccati-type equation. Then, the results for  \eqref{eq:intro_uno}
are derived by analyzing 
the asymptotic behavior of the Riccati equation. 
For instance in \cite{bellman_1950}, Bellman present
the analysis of the second 
order differential equation $u^{(2)}-(1+f(t))u=0$ by
introducing the new variable $v=u^{(1)}/u$ which transform the  linear perturbed equation
in the following Riccati equation $v^{(1)}+v-(1+f(t))=0$.
Then, by assuming several conditions on the regularity and integrability  of $f$,
he obtains the formulas for characterization of
the asymptotic behavior of $u$. For example in the case that
$f(t)\to 0$ when $t\to\infty$, Bellman proves
that there exists two linearly independent 
solutions $u_1$ and $u_2$, such that
$(u^{(1)}_i/u_i)(t)\to (-1)^{i+1}$ when $t\to\infty$ and
\begin{eqnarray*}
 \exp\Big((-1)^{i+1}t-\int_{t_0}^t|f(\tau)|d\tau\Big)\le u_i(t)
 \le \exp\Big((-1)^{i+1}t+\int_{t_0}^t|f(\tau)|d\tau\Big)
 \quad\mbox{for $i=1,2$.}
\end{eqnarray*}
More details and a summarization of the 
results of the application of the scalar method to a special second order equation
are given in \cite{bellman_book}. 

Let us recall some classical results. The list of the results is  non-exhaustive.
Firstly, we recall that Poincar\'e, in \cite{poincare_1885} assumes two hypothesis:
\begin{enumerate}
 \item[(P$_1$)] $\lambda$ is a simple characteristic root of \eqref{eq:intro_dos} 
 distinct of the
real part of the any other characteristic root
 \item[(P$_2$)] the perturbation
functions $r_j$ are rational functions such that, for all $j=0,\ldots,3$,
$r_j(t)\to 0$ when $t\to\infty$.
\end{enumerate}
Then, under (P$_1$)-(P$_2$), he deduce that
$y(t)$, the solution  of \eqref{eq:intro_uno}, has the following asymptotic
behavior: $y^{(\ell)}(t)/y(t)\to (\lambda)^{\ell}$
for $\ell=1,2,3,4$ when $t\to\infty$.
Afterwards, Perron \cite{perron_1909} extends the results of Poincar\'e by 
assuming (P$_1$) and considering instead of (P$_2$) the hypothesis that
the perturbation functions $r_j$ are continuous functions such that, for all $j=0,\ldots,3$,
$r_j(t)\to 0$ when $t\to\infty$. Perhaps, other three important landmarks on the asymptotic
behavior are the contributions of Levinson \cite{levinson1948}, 
Hartman-Wintner \cite{hartman_wintner_1955} and Harrris and Lutz
\cite{harris_lutz_1974,harris_lutz_1977}.
In \cite{levinson1948}, 
Levinson analyze the non-autonomous system $\mathbf{x}'(t)=[\Lambda(t)+R(t)]\mathbf{x}(t)$
where $\Lambda$ is a diagonal matrix and $R$ is the perturbation matrix.
Levinson assumes that the diagonal matrix satisfies a dichotomy condition
and the perturbation function is continuous and belongs to $L^1([t_0,\infty[)$
and prove that a fundamental matrix $X$ has the following
asymptotic representation $X(t)=[I+o(1)]\exp\Big(\int_{t_0}^t\Lambda(s)ds\Big)$.
Meanwhile, Hartman-Wintner assumes that 
the diagonal matrix satisfies a more strong condition than
the Levinson dichotomy condition
and the perturbation function is continuous and belongs to $L^p([t_0,\infty[)$
for some $p\in ]1,2]$
and prove that 
$X(t)=[I+o(1)]\exp\Big(\int_{t_0}^t(\Lambda(s)+\mathrm{diag}(R(s)))ds\Big)$.
In the seventies Harrris and Lutz in 
\cite{harris_lutz_1974} 
(see also \cite{harris_lutz_1977,figueroa_2006,pfeiffer_1972,pfeiffer_1970,pinto_2003})
find a change of variable to unify the results of Levinson and Hartman-Wintner.
Other, important contributions are given for instance by 
\cite{simsa_1988,uri_harry}. We comment that the application of 
Levinson and Hartman-Wintner results to \eqref{eq:intro_uno} are not direct and
should be done via the nonanalytic theory. Here, the main practical 
disadvantage is that,
in most of the cases  $\Lambda$ and $R$ are difficult to
algebraic manipulation and the asymptotic formulas are only theoretical ones.

In this paper, we 
reorganize and reformulate 
the original scalar method of Bellman and then
introduce new hypothesis in order to characterize the asymptotic behavior
of the solutions for \eqref{eq:intro_uno}
by considering
that the perturbation functions satisfy some restrictions
in the $L^p$ sese.
Indeed, the  scalar method is presented in three big steps
(see section~\ref{sec:mainres}).
First, we introduce a change of variable and deduce that
the new variable is a solution of a Riccati-type equation.
In a second step, in order to deduce the well posedness and the
asymptotic behavior of the solution for the Riccati-type equation,
we assume a general hypothesis about 
the linear part of \eqref{eq:intro_uno} and the perturbation
functions.
Then, in a third step, we translate the results 
for the solution of the Riccati-type equation to the solution of
\eqref{eq:intro_uno}. In this step,
we deduce the existence of a fundamental system of
solutions for \eqref{eq:intro_uno} and conclude the process with
the formulation and proof of the asymptotic integration 
formulas for the solutions of \eqref{eq:intro_uno}.

The main results of the paper are summarized
in section~\ref{sec:mainres}.
These  results are obtained by considering the following hypothesis
about the coefficients and perturbation functions
of \eqref{eq:intro_uno}
\begin{enumerate}
 \item[(H$_1$)] All roots of the characteristic polynomial 
 $\lambda^{4}+\sum_{i=0}^3a_i \lambda^{i}$, associated to 
 \eqref{eq:intro_dos}, has distinct real part or equivalently 
  $\Big\{\lambda_i, i=\overline{1,4}\;:\;
 \lambda_1>\lambda_2>\lambda_3>\lambda_4\Big\}\subset\R$
 is  the set of characteristic  roots for  \eqref{eq:intro_dos}.

 \item[(H$_2$)] For all $j=0,\ldots,3,$ the perturbation functions $r_j$
 are selected such that 
 $\mathcal{L}(r_j)(t)\to 0$ when $t\to\infty$, 
where $\mathcal{L}$ is the functional on $L^p([t_0,\infty[)$  defined as follows
\begin{eqnarray}
  \mathcal{L}(E)(t) 
  &=& \int_{t_0}^{\infty}\left[|g(t,s)|
	+\left|\frac{\partial g}{\partial t}(t,s)\right|
	+\left|\frac{\partial^2 g}{\partial t^2}(t,s)\right|
	\right]|E(s)|ds.
\label{eq:functyonal_L}
\end{eqnarray}
Here $L^p([t_0,\infty[)$ is the space of measurable functions 
on $[t_0,\infty[$ for some $t_0>0,$ such that
are $p$-integrable in the sense of Lebesgue
for $p\in [1,\infty[$ or essentially bounded 
for $p=\infty$. 
 
 \item[(H$_3$)] For all $j=0,\ldots,3,$
 the perturbation functions $r_j$
 are belong to
 $\mathcal{F}_{\rho_i}([t_0,\infty[)$ defined by
\begin{eqnarray*}
\mathcal{F}_{\rho_i}([t_0,\infty[)
	&=&\Bigg\{E:[t_0,\infty[\to\R
	\; :\;
	\mathbb{F}_{i}(E)(t)\le \rho_i,
	\quad
	t\ge t_0
	\;\;
	\Bigg\},
\end{eqnarray*}
for each $i=1,\ldots,4,$ with
$\rho_i\in[\mathbb{F}_i(1)(t),\infty[\subset\mathbb{R}^+$ a given (fix) number and
the operators $\mathbb{F}_{i}$ defined as follows
\begin{eqnarray*}
\mathbb{F}_1(E)(t)
	&=&
	\int_{t}^{\infty}e^{-(\lambda_2-\lambda_1)(t-s)}|E(s)|ds,
\\
\mathbb{F}_2(E)(t)
	&=&
	\int_{t_0}^{t}e^{-(\lambda_1-\lambda_2)(t-s)}|E(s)|ds+
	\int_{t}^{\infty}e^{-(\lambda_3-\lambda_2)(t-s)}|E(s)|ds,
\\
\mathbb{F}_3(E)(t)
	&=&
	\int_{t_0}^{t}e^{-(\lambda_2-\lambda_3)(t-s)}|E(s)|ds+
	\int_{t}^{\infty}e^{-(\lambda_4-\lambda_3)(t-s)}|E(s)|ds,
\\
\mathbb{F}_4(E)(t)
	&=&
	\int_{t_0}^{t}e^{-(\lambda_3-\lambda_4)(t-s)}|E(s)|ds.
\end{eqnarray*}

\end{enumerate}
We should be comment that (H$_1$)-(H$_2$)
are used to prove the existence of a fundamental system of solutions for
\eqref{eq:intro_uno} and  (H$_3$) is needed in order to get 
the asymptotic behavior formulas for solutions of \eqref{eq:intro_uno}.
The hypothesis (H$_2$) is new and is the natural generalization of the 
classical hypothesis introduced by Poincar\'e when the perturbation functions
are integrable functions instead of
rational functions. We note that a similar hypothesis to (H$_2$) was introduced
by Figueroa and Pinto~\cite{figueroa_2008}.

The paper is organized as follows. In section~\ref{sec:mainres}
we present the reformulated scalar method and the main results of this paper.
Then, in section~\ref{sec:proofs}
we present the proofs of Theorems~\ref{teo:solution_riccati_type_general},
\ref{teo:solution_asymptotic} and \ref{teo1.7.1}.

\section{Revisited Bellman method and main results}
\label{sec:mainres}

In this section we present the scalar method as a process of three steps.
In each step we present the main results which proofs are deferred to
section~\ref{sec:proofs}.

\subsection{Step 1: Change of variable and reduction of the order.}
We introduce a little bit different change of variable
to those  originally proposed  by Bellman. Here, in this paper, the new variable $z$
is of the following type 
\begin{eqnarray}
z(t)=\frac{y^{(1)}(t)}{y(t)}-\mu 
\quad\mbox{or equivalently}\quad
y(t)=\exp\Big(\int_{t_0}^t (z(s)+\mu )ds\Big),
\label{eq:general_change_var}
\end{eqnarray}
where $y$ is a solution of \eqref{eq:intro_uno} and
$\mu$ is an arbitrary root of the characteristic polynomial associated 
to \eqref{eq:intro_dos}. 
Then, by differentiation of $y(t)$ and
by replacing the results of $y^{\ell}(t)$, $\ell=0,\ldots,4$, 
in \eqref{eq:intro_uno}, we deduce that $z$ is a solution of 
the following third order Riccati-type
\begin{eqnarray}
&&z^{(3)}+[4\mu  +a_3]z^{(2)}
    +[6\mu ^2+3a_3\mu  +a_2]z^{(1)}
    +[4\mu  ^3+3\mu  ^2a_3+2\mu  a_2
    +a_1]z 
    \nonumber\\
   & & \qquad
   +r_3z^{(2)}+[3\mu  r_3(t)+r_2(t)]z^{(1)}
   +[3\mu  ^2r_3(t)+2\mu  r_2(t)+r_1(t)]z
   +\mu  ^3r_3(t) 
   \nonumber\\
   & & \qquad
   +\mu  ^2r_2(t)+\mu  r(t)+r_0(t)
   +4zz^{(2)}+[\mu    +3a_3+3r_3(t)]zz^{(1)}
   +6z^2z^{(1)} 
   \nonumber\\
   & &\qquad
   +3[z^{(1)}]^2+[6\mu  ^2+3\mu  a_3+a_2
   +3\mu  r_3(t)+r_2(t)]z^2+[4\mu  +r_3(t)]z^3+z^4
   =0.
   \label{eq:ricati_original}
\end{eqnarray}
Now, if we define the operators $\Uppsi^l$ and $\Uppsi^n$
by the following relations
\begin{eqnarray}
\Uppsi^l(\mu,h)&=&h^{(3)}+[4\mu  +a_3]h^{(2)}+[6\mu^2+3a_3\mu 
    +a_2]h^{(1)}
   \nonumber\\   
    & & +[4\mu  ^3+3\mu  ^2a_3+2\mu  a_2+a_1]h,
 \label{eq:riccati_type_lineal}   
\\
\Uppsi^n(\mu,h)  &= & r_3h^{(2)}+[3\mu  r_3+r_2]h^{(1)}+[3\mu 
   ^2r_3+2\mu  r_2+r_1]h+\mu  ^3r_3 
   \nonumber\\
   & & +\mu  ^2r_2+\mu  r_i+r_0+4hh^{(2)}
   +[\lambda\mu  +3a_3+3r_3]hh^{(1)}+6h^2h^{(1)} 
   \nonumber\\
   & &
   +3[h^{(1)}]^2+[6\mu  ^2+3\mu  a_3+a_2
   +3\mu  r_3+r_2]h^2+[4\mu  +r_3]h^3+h^4,
   \label{eq:riccati_type_nonlineal}
\end{eqnarray}
we note that \eqref{eq:ricati_original} can be 
equivalently rewritten as follows
\begin{eqnarray}
\Uppsi^l (\mu,z)+\Uppsi^n (\mu,z)=0.
\label{eq:riccati_type}
\end{eqnarray}
Note that $\Uppsi^l$ and  $\Uppsi^n$ are linear and
nonlinear operators, respectively.
Thus, the analysis of original linear perturbed
equation of fourth order 
\eqref{eq:intro_uno} is translated to the analysis of a nonlinear
third order equation  \eqref{eq:riccati_type}.
Moreover, 
we note that characteristic polynomials associated to \eqref{eq:intro_dos} 
and to $\Uppsi^l(\lambda,z)=0$ are related in the sense of the following Proposition.

\begin{proposition}
\label{lem:solucion_partelineal}
Let us consider  $\Uppsi^l$ the operator defined in \eqref{eq:riccati_type_lineal}.
If  $\lambda_i$   and $\lambda_j$ are  two distinct  
characteristic polynomials associated to \eqref{eq:intro_dos},
then $\lambda_j-\lambda_i$ is a 
root of the characteristic polynomial associated to 
the differential equation $\Uppsi^l (\lambda_i,z)=0$. 
\end{proposition}

\begin{proof}
Considering $\lambda_i\not = \lambda_j$ satisfying the characteristic polynomial associated to 
\eqref{eq:intro_dos}, subtracting the equalities, dividing the result
by $\lambda_j-\lambda_i$ and using the identities
\begin{eqnarray*}
  \lambda_j^3+\lambda_j^2\lambda_i+\lambda_j\lambda_i^2+\lambda_i^3 &=& (\lambda_j-\lambda_i)^3+4\lambda_i(\lambda_j-\lambda_i)^2+6\lambda_i^2(\lambda_j-\lambda_i)+4\lambda_i^3 \\
  a_3(\lambda_j^2+\lambda_j\lambda_i+\lambda_i^2) &=& a_3(\lambda_j-\lambda_i)^2+3a_3\lambda_i(\lambda_j-\lambda_i)+3a_3\lambda_i^2 \\
  a_2(\lambda_j-\lambda_i) &=&
  a_2(\lambda_j-\lambda_i)+2\lambda_ia_2,
\end{eqnarray*}
we deduce that $\Uppsi^l (\lambda_j-\lambda_i,z)=0$.
Thus, $\lambda_j-\lambda_i$ is a root of the characteristic polynomial associated to 
$\Uppsi^l (\lambda_i,z)=0$ and the proof is concluded.
\end{proof}

We note that the change of variable \eqref{eq:general_change_var} can be applied
by each characteristic root $\lambda_i$ and the equation \eqref{eq:riccati_type}
should be satisfied with $\mu=\lambda_i$. Then, in order to distinguish that $z$
is a solution of \eqref{eq:riccati_type} with $\mu=\lambda_i$ we introduce the notation
$z_i$. Hence, to conclude this step we precise the previous discussion in the
following Lemma. 

\begin{lemma}
\label{lem:cambio_var}
If hypothesis (H$_1$) is satisfied, 
then the fundamental system of solutions of 
\eqref{eq:intro_uno} is given by
\begin{eqnarray}
    y_i(t)=\exp\Big(\int_{t_0}^{t}[\lambda_i+z_i(s)]ds\Big),
    \quad
    \mbox{with $\{\mu_i,z_i\}$ solution of \eqref{eq:riccati_type},}
    \quad i\in\{1,2,3,4\}.
    \label{eq:fundam_sist_sol}
\end{eqnarray}
\end{lemma}

\subsection{Step 2: Well posedness and asymptotic behavior 
of the Riccati-type equation~\eqref{eq:ricati_original}.}
In this second step, we obtain three results.
The first result is related to the conditions
for the existence and uniqueness 
of a more general  equation of that given in \eqref{eq:ricati_original},
see Theorem~\ref{teo:solution_riccati_type_general}.
Then, we introduce a second result concerning to the well posedness 
of~\eqref{eq:ricati_original}, see Theorem~\ref{lem:solution_perturbation}. 
Finally, we present the result of asymptotic behavior for
\eqref{eq:ricati_original}, see Theorem~\ref{teo:solution_asymptotic}. 
Indeed, to be precise these three results are the following theorems:
\begin{theorem}
\label{teo:solution_riccati_type_general}
Let us introduce the notation $C_0^2([t_0,\infty[)$
for the following space of functions
\begin{eqnarray*}
C_0^2([t_0,\infty[)
=\Big\{z\in C^2([t_0,\infty[,\R)\quad :\quad
z,z^{(1)},z^{(2)}\to 0\text{ when }t\to\infty\Big\},
\quad t_0\in\R,
\end{eqnarray*}
and consider the equation 
\begin{eqnarray}
z^{(3)}+ \sum_{i=0}^2 b_i z^{(i)}=\Omega(t)+F(t,z,z^{(1)},z^{(2)}),
\quad 
(b_0,b_1,b_2)\in\mathbb{R}^3,
\label{eq:riccati_type_general}
\end{eqnarray}
where $\Omega$ and $F$ are given functions such that the following
restrictions
\begin{enumerate}
  \item[($\mathcal{R}_1$)] There exists the functions 
  $\hat{F}_1,\hat{F}_2,\Gamma:\mathbb{R}^4\to\mathbb{R}$; 
  $\Lambda_1,\Lambda_2:\mathbb{R}\to\mathbb{R}^3$ and $\mathbf{C}\in\R^7$, such that 
\begin{eqnarray*}  
 F  &=&\hat{F}_1+\hat{F}_2+\Gamma,
\\
\hat{F}_1(t,x_1,x_2,x_3)&=&\Lambda_1(t)\cdot (x_1,x_2,x_3),
\\ 
\hat{F}_2(t,x_1,x_2,x_3)&=&\Lambda_2(t)\cdot (x_1x_2,x^2_1,x^3_1),
\\ 
\Gamma(t,x_1,x_2,x_3)&=&\mathbf{C}\cdot 
(x^2_2,x_1x_2,x_1x_3,x^2_1,x^2_1x_2,x^3_1,x^4_1),
\end{eqnarray*}
where ``$\cdot$'' denotes the  canonical inner product in $\R^n$.

\item[($\mathcal{R}_2$)]  The roots $\gamma_i$, $i=1,2,3,$
of the corresponding characteristic
polynomial associated to the homogeneous part of \eqref{eq:riccati_type_general}
are real and simple.

\item[($\mathcal{R}_3$)] It is assumed that
$\mathcal{L}(\Omega)(t)\to 0$, $\mathcal{L}(\|\Lambda_1\|_1)(t)\to 0$
and
$\mathcal{L}\Big(\|\Lambda_2\|_1\Big)(t)$ 
is bounded, when $t\to\infty$. 
Here $\|\cdot\|_1$ denotes the norm of the sum in $\R^n$
and  $\mathcal{L}$ is the operator defined on~\eqref{eq:functyonal_L}.
\end{enumerate}
hold. Then, there exists a unique  $z\in C_0^2([t_0,\infty[)$ solution 
of  \eqref{eq:riccati_type_general}.
\end{theorem}

\begin{theorem}
\label{lem:solution_perturbation}
Let us consider that the hypothesis (H$_1$) and (H$_2$) are satisfied.
Then, for each $i=1,\ldots,4$, the equation~\eqref{eq:riccati_type} has a unique solution
$\{\mu_i,z_i\}$ with
$z_i\in C_0^2([t_0,\infty[)$.
\end{theorem}

\begin{theorem}
\label{teo:solution_asymptotic}
Consider that the hypothesis (H$_1$),(H$_2$) and (H$_3$) 
are satisfied and for $i=1,\ldots,4,$  
introduce the notation
\begin{eqnarray*}
A_i&=&\frac{1}{|\delta w_i|}\sum_{j=0}^2\alpha_{j,i}
\quad
\mbox{and}
\\
\upvarsigma_i&=&3|\lambda_i|^2+5|\lambda_i|+3+
        \Big(19+7|\lambda_i|+|12\lambda_i+3a_3|+
        |6\lambda^2_i+3\lambda_ia_3+a_2|\Big)\upeta,
\quad \upeta\in ]0,1/2[,
\end{eqnarray*}
 with
\begin{eqnarray*}
 \delta w_1&=&(\lambda_3-\lambda_2)(\lambda_4-\lambda_3)(\lambda_4-\lambda_2),
 \quad
 \delta w_2=(\lambda_3-\lambda_1)(\lambda_4-\lambda_3)(\lambda_4-\lambda_1),
 \\
 \delta w_3&=&(\lambda_2-\lambda_1)(\lambda_4-\lambda_2)(\lambda_4-\lambda_1),
 \quad
 \delta w_4=(\lambda_2-\lambda_1)(\lambda_3-\lambda_2)(\lambda_3-\lambda_1),
 \\
 \alpha_{j,1}
 &=&
 	|\lambda_4-\lambda_3||\lambda_2-\lambda_1|^j+
	|\lambda_4-\lambda_2||\lambda_3-\lambda_1|^j+
	|\lambda_3-\lambda_2||\lambda_4-\lambda_1|^j
\\
\alpha_{j,2}
 &=&
	|\lambda_3-\lambda_4||\lambda_1-\lambda_2|^j+
	|\lambda_1-\lambda_3||\lambda_4-\lambda_2|^j+
	|\lambda_1-\lambda_4||\lambda_3-\lambda_2|^j
 \\ 
 \alpha_{j,3}
 &=&
 	|\lambda_2-\lambda_1||\lambda_4-\lambda_3|^j+
	|\lambda_2+\lambda_4-2\lambda_3||\lambda_1-\lambda_3|^j+
	|\lambda_1+\lambda_4-2\lambda_3||\lambda_2-\lambda_3|^j
 \\ 
 \alpha_{j,4}
 &=&
	|\lambda_3-\lambda_2||\lambda_1-\lambda_4|^j+
	|\lambda_3-\lambda_1||\lambda_2-\lambda_4|^j+
	|\lambda_2-\lambda_1||\lambda_3-\lambda_4|^j 
\end{eqnarray*}
For each $i=1,\ldots,4,$
if $\rho_i \in ]0,(A_{i}\upvarsigma_{i})^{-1}[$, 
then $\{\mu_i,z_i\}$, the solution of \eqref{eq:riccati_type},
has the following asymptotic behavior 
\begin{eqnarray}
z^{(\ell)}_{i}(t)
=
\left\{
\begin{array}{lll}
\displaystyle
O\Big(\int_{t}^{\infty}e^{-\beta(t-s)}|p(\lambda_1,s)|ds\Big),
&\qquad
& i=1,\quad \beta\in [\lambda_2-\lambda_1,0[,
\\
\displaystyle
O\Big(\int_{t_0}^{\infty}e^{-\beta(t-s)}|p(\lambda_2,s)|ds\Big),
&\qquad
& i=2,\quad \beta\in [\lambda_3-\lambda_2,0[,
\\
\displaystyle
O\Big(\int_{t_0}^{\infty}e^{-\beta(t-s)}|p(\lambda_3,s)|ds\Big),
&\qquad
& i=3,\quad \beta\in [\lambda_4-\lambda_3,0[,
\\
\displaystyle
O\Big(\int_{t_0}^{t}e^{-\beta(t-s)}|p(\lambda_4,s)|ds\Big),
&\qquad
& i=4,\quad \beta\in ]0,\lambda_3-\lambda_4],
\end{array}
\right.
\label{eq:asymptotic_teo:form}
\end{eqnarray}
where $p(\mu,s)=\mu r^3_3(s)+\mu r^2_2(s)+\mu r_1(s)+r_0(s).$
\end{theorem}

\subsection{Step 3: Existence of a fundamental system of solutions for \eqref{eq:intro_uno}
and its asymptotic behavior.}
Here we translate the results for the behavior of $z$   
(see Theorem~\ref{lem:solution_perturbation})
to the variable $y$ via the relation~\eqref{eq:general_change_var}.

\begin{theorem}
\label{teo1.7.1}
Let us assume that the hypothesis (H$_1$) and (H$_2$) are satisfied,
denote by $W[y_1,\ldots,y_4]$ the Wronskian of $\{y_1,\ldots,y_4\}$, by
\begin{eqnarray*}
\pi_i= \prod_{k\in N_i}(\lambda_k-\lambda_i),
\quad 
N_i=\{1,2,3,4\}-\{i\},
\quad
i=1,\ldots,4,
\end{eqnarray*}
by $p(\mu,s)$ the function defined in Theorem~\ref{teo:solution_asymptotic}
and by $F$ the functions defined in Theorem~\ref{teo:solution_riccati_type_general}
with $\Lambda_1,\Lambda_2$ and $\mathbf{C}$ given in \eqref{eq:notation_ricc}.
Then the equation \eqref{eq:intro_uno} has a fundamental system of 
solutions given by \eqref{eq:fundam_sist_sol}.
 Moreover the following properties about
 the asymptotic behavior
\begin{eqnarray}
\frac{y_i^{(\ell)}(t)}{y_i(t)}&=&\Big(\lambda_i\Big)^{\ell},
\quad\mbox{for $i\in\{1,2,3,4\}$ and $\ell\in\{1,2,3,4\}$,}
\label{eq:asymptotic_perturbed}
\\
W[y_1,\ldots,y_4]&=& (\lambda_4-\lambda_1)(\lambda_3-\lambda_1)
(\lambda_2-\lambda_1)(\lambda_3-\lambda_2)
(\lambda_4-\lambda_2)(\lambda_4-\lambda_3),
\label{eq:asymptotic_perturbed_w}
\end{eqnarray}
are satisfied when $t\to\infty.$
Furthermore, if (H$_3$) is satisfied, then 
\begin{eqnarray}
y_i(t)&=& e^{\lambda_i(t-t_0)}
	\exp\Big(\pi^{-1}_i\int_{t_0}^{t}
	\Big[p(\lambda_i,s)+F(s,z_i(s),z_i^{(1)}(s),z_i^{(2)}(s))\Big]ds\Big),
\label{eq:asymptotic_perturbed_2}
\\
y_i^{(\ell)}(t)&=&
	\Big(\lambda_i^{\ell-1}+o(1)\Big)e^{\lambda_i(t-t_0)}
\nonumber\\
&&	\times
	\exp\Big(\pi^{-1}_i\int_{t_0}^{t}
	\Big[p(\lambda_i,s)+F(s,z_i(s),z_i^{(1)}(s),z_i^{(2)}(s))\Big]ds
	\Big),
	\;\;\ell=\overline{2,4},
\label{eq:asymptotic_perturbed_3}
\end{eqnarray}
holds, when $t\to\infty$ with $z_i$ given asymptotically by \eqref{eq:asymptotic_teo:form}.
\end{theorem}

\section{Proof of main results}
\label{sec:proofs}

In this section we present the proofs of 
Theorems~\ref{teo:solution_riccati_type_general},
\ref{lem:solution_perturbation},
\ref{teo:solution_asymptotic},
and \ref{teo1.7.1}.

\subsection{Proof of Theorem~\ref{teo:solution_riccati_type_general}}

Before of start the proof, we need define 
some notation about Green functions. First, let us consider the homogeneous equation associated
to \eqref{eq:riccati_type_general}
\begin{eqnarray}
z^{(3)}+ \sum_{i=0}^2 b_i z^{(i)}=0,
\label{eq:riccati_type_equiv_homo}
\end{eqnarray}
and denote by $\gamma_i,\quad i=1,2,3,$ the roots of the corresponding characteristic
polynomial for \eqref{eq:riccati_type_equiv_homo}. Then, the green function for 
\eqref{eq:riccati_type_equiv_homo} is defined by 
\begin{eqnarray}
g(t,s)=\frac{1}{\delta\gamma}
\times
\left\{
\begin{array}{lcl}
 g_1(t,s), 
	& \quad & (\re\gamma_1,\re\gamma_2,\re\gamma_3)\in \mathbb{R}^3_{---},
 \\
 g_2(t,s), 
	& & (\re\gamma_1,\re\gamma_2,\re\gamma_3)\in \mathbb{R}^3_{+--},
 \\
 g_3(t,s), 
	& & (\re\gamma_1,\re\gamma_2,\re\gamma_3)\in \mathbb{R}^3_{++-},
 \\
 g_4(t,s), 
	& & (\re\gamma_1,\re\gamma_2,\re\gamma_3)\in \mathbb{R}^3_{+++}, 
\end{array}
\right.
\label{eq:green_function}
\end{eqnarray} 
where $\delta\gamma=(\gamma_2-\gamma_1)(\gamma_3-\gamma_2)(\gamma_3-\gamma_1) $
and
\begin{eqnarray*}
g_1(t,s)&=&
\left\{
\begin{array}{lcl}
 0, & &t\ge s,
 \\ 
(\gamma_3-\gamma_2)e^{-\gamma_1(t-s)}
      +(\gamma_1-\gamma_3)e^{-\gamma_2(t-s)}
      +(\gamma_2-\gamma_1)e^{-\gamma_3(t-s)},
      & \quad & t\le s,
\end{array}
\right.
\\
g_2(t,s)&=&
\left\{
\begin{array}{lcl}
(\gamma_1-\gamma_2)e^{-\gamma_3(t-s)}
      -(\gamma_1-\gamma_3)e^{-\gamma_2(t-s)},
      & \quad & t\le s,
 \\
(\gamma_2-\gamma_3)e^{-\gamma_1(t-s)},
      & &t\ge s,
\end{array}
\right.
\\
g_3(t,s)&=&
\left\{
\begin{array}{lcl}
(\gamma_2-\gamma_1)e^{-\gamma_3(t-s)},
      & &t\le s,
  \\     
(\gamma_3-\gamma_2)e^{-\gamma_1(t-s)}
      -(\gamma_3-\gamma_1)e^{-\gamma_2(t-s)},
      & \quad & t\ge s,
\end{array}
\right.
\\
g_4(t,s)&=&
\left\{
\begin{array}{lcl}
(\gamma_3-\gamma_2)e^{-\gamma_1(t-s)}
      +(\gamma_1-\gamma_3)e^{-\gamma_2(t-s)}
      +(\gamma_2-\gamma_1)e^{-\gamma_3(t-s)},
      & \quad & t\ge s,
 \\
 0, & &t\le s.
\end{array}
\right.
\end{eqnarray*}
Further details on Green functions may be consultd in~\cite{bellman_book}.

Now, we start the proof be noticing that, 
by the method of variation of parameters,  the hypothesis $(\mathcal{R}_2)$, 
implies that the equation
\eqref{eq:riccati_type_general} is equivalent to the following integral equation
\begin{eqnarray}
z(t)=\int_{t_0}^\infty g(t,s)\Big[\Omega(s)
+F\Big(s,z(s),z^{(1)}(s),z^{(2)}(s)\Big)\Big]ds,
\label{eq:integ_equation}
\end{eqnarray}
where $g$ is the Green function defined on  \eqref{eq:green_function}.
We recall that $C_0^2([t_0,\infty[)$ is a Banach space with the norm 
$\|z\|_0=\sup_{t\geq t_0}\{|z(t)|+|z^{(1)}(t)|+|z^{(2)}(t)|\}.$
Now, we define the operator $T$ from $C_0^2([t_0,\infty[)$ to 
$ C_0^2([t_0,\infty[)$ as follows
\begin{eqnarray}
Tz(t)&=&\int_{t_0}^\infty g(t,s)\Big[\Omega(s)
+F\Big(s,z(s),z^{(1)}(s),z^{(2)}(s)\Big)\Big]ds.
\label{eq:operator_fix_point}
\end{eqnarray}
Then, we note that \eqref{eq:integ_equation} can be rewritten as the operator
equation 
\begin{eqnarray}
Tz=z
\qquad
\mbox{over}
\qquad
D_{\upeta}:=\Big\{z\in C_0^2([t_0,\infty[)\quad:\quad \|z\|_0\le\upeta \Big\},
\label{eq:operator_equation}
\end{eqnarray}
where $\upeta\in\R^+$ will be selected
in order to apply the Banach fixed point theorem.
Indeed, we have that

\vskip 0.5cm
\noindent
{\bf (a)} {\it $T$ is well defined from $C_0^2([t_0,\infty[)$ to $C_0^2([t_0,\infty[)$}.
Let us consider an arbitrary $z\in C_0^2([t_0,\infty[).$ We note that
 \begin{eqnarray*}
  T^{(i)}z(t) &=& 
  \int_{t_0}^{\infty}
	\frac{\partial^i g}{\partial t^i}(t,s)
	\Big[\Omega(s)
	+F\Big(s,z(s),z^{(1)}(s),z^{(2)}(s)\Big)
	\Big]ds,
	\quad i=1,2.
\end{eqnarray*}
Then, by the definition of $g$, we immediately deduce that 
$Tz,T^{(1)}z,T^{(2)}z\in C^2([t_0,\infty[,\R).$
Furthermore,  by the hypothesis $(\mathcal{R}_1)$, we can deduce the following estimate
\begin{eqnarray}
  |T^{(i)}z(t)| &\le & \int_{t_0}^{\infty}
	\left|\frac{\partial^i g}{\partial t^i}(t,s)\right|\Big[
	|\Omega(s)|
	+\Big|\hat{F}_1\Big(s,z(s),z^{(1)}(s),z^{(2)}(s)\Big)\Big| 
	\nonumber\\
	&&
	\qquad
	+\Big|\hat{F}_2\Big(s,z(s),z^{(1)}(s),z^{(2)}(s)\Big)\Big| 
	+\Big|\Gamma\Big(s,z(s),z^{(1)}(s),z^{(2)}(s)\Big)\Big|
	\Big]ds
	\label{eq:bound_for_tz0}
\end{eqnarray}
for each $i=0,1,2$.
Now, by application of the hypothesis $(\mathcal{R}_3)$,
the properties of $\hat{F}_1,\hat{F}_2$ and $\Gamma$
and the fact that  $z\in C^2_0$, we have that
the right hand side of \eqref{eq:bound_for_tz0} tends to $0$
when $t\to\infty$. Then,
$Tz,T^{(1)}z,T^{(2)}z\to 0$ when $t\to\infty$
or equivalently $Tz\in C^2_0$ for all $z\in C^2_0.$

\vskip 0.5cm
\noindent
{\bf (b)} {\it For all $\upeta\in ]0,1[$, the set $D_{\upeta}$ 
is invariante under $T$}. Let us consider $z\in D_{\upeta}$.
From \eqref{eq:bound_for_tz0}, we can deduce the following estimate
\begin{eqnarray}
\|Tz\|_0
&\le& \mathcal{L}(\Omega)(t)
	+\|z\|_0 
	\mathcal{L}\Big(\|\Uplambda_1\|_1\Big)(t)
	+2\Big(\|z\|_0\Big)^2\mathcal{L}\Big(\|\Uplambda_2\|_1\Big)(t)
	+\Big(\|z\|_0\Big)^3\mathcal{L}\Big(\|\Uplambda_2\|_1\Big)(t)
	\nonumber\\
	&&
	\hspace{0.8cm}
	+\Big(\|z\|_0\Big)^2
	\Bigg(\sum_{i=1}^4|c_i|+\Big(|c_5|+|c_6|\Big)\|z\|_0+|c_7|\Big(\|z\|_0\Big)^2\Bigg)
	\mathcal{L}\Big(1\Big)(t)
\nonumber
\\
&\le&  I_1+I_2.
\label{eq:bound_for_tz0_inv}
\end{eqnarray}
where
\begin{eqnarray*}
 I_1&=&\mathcal{L}(\Omega)(t)
\\
I_2&=&\|z\|_0
	\Bigg\{
	\mathcal{L}\Big(\|\Uplambda_1\|_1\Big)(t)
	+\Bigg(
	2\mathcal{L}\Big(\|\Uplambda_2\|_1\Big)(t)
	+\mathcal{L}\Big(\|\mathbf{C}\|_1\Big)(t)\Bigg)\|z\|_0
\nonumber
\\
&&
	\qquad
	+\Bigg(\mathcal{L}\Big(\|\Uplambda_2\|_1\Big)(t)+
	\mathcal{L}\Big(\|\mathbf{C}\|_1\Big)(t)
	\Bigg)\Big(\|z\|_0\Big)^2
	+\mathcal{L}\Big(\|\mathbf{C}\|_1\Big)(t)\Big(\|z\|_0\Big)^3\Bigg\}.
\end{eqnarray*}
Now, by $(\mathcal{R}_3)$ we deduce that 
$I_1\to 0$ when $t\to\infty$. Similarly, by application
of $(\mathcal{R}_3)$ we can prove that the inequality
\begin{eqnarray*}
 I_2&\le&  
	\upeta^2
	\Bigg\{
	\Bigg(
	2\mathcal{L}\Big(\|\Uplambda_2\|_1\Big)(t)
	+\mathcal{L}\Big(\|\mathbf{C}\|_1\Big)(t)\Bigg)
	+\Bigg(\mathcal{L}\Big(\|\Uplambda_2\|_1\Big)(t)+
	\mathcal{L}\Big(\|\mathbf{C}\|_1\Big)(t)
	\Bigg)\upeta
\nonumber
\\
&&
	\qquad
	+\mathcal{L}\Big(\|\mathbf{C}\|_1\Big)(t)\upeta^2\Bigg\}
\nonumber
\\ 
&\le& \upeta,
\end{eqnarray*}
holds
when $t\to\infty$ in a right neighborhood of $\upeta=0$. 
Hence, by \eqref{eq:bound_for_tz0_inv} 
and $(\mathcal{R}_3)$,
we prove that $Tz\in D_{\upeta}$ for all $z\in D_{\upeta}$.

\vskip 0.5cm
\noindent
{\bf (c)} {\it $T$ is a contraction for $\upeta\in ]0,1/2[$}. 
Let $z_1,z_2\in D_{\upeta}$, by the hypothesis $(\mathcal{R}_1)$
and algebraic rearrangements,
we follow that
\begin{eqnarray*}
  \|T z_1-T z_2\|_0
   &\leq& 
   \|z_1-z_2\|_0\sum_{i=0}^2
   \int_{t_0}^{\infty}
   \left|\frac{\partial^i g}{\partial t^i}(t,s)\right|
   \|\Uplambda_1(s)\|_1ds
   \\
   &&
   \quad+
   \|z_1-z_2\|_0\max\Big\{2\upeta,3\upeta^2\Big\}
   \sum_{i=0}^2
   \int_{t_0}^{\infty}
   \left|\frac{\partial^i g}{\partial t^i}(t,s)\right|
   \|\Uplambda_2(s)\|_1ds 
   \\
   & &    
   \quad+
   \|z_1-z_2\|_0\max\Big\{2\upeta,3\upeta^2,4\upeta^3\Big\}
   \sum_{i=0}^2
   \int_{t_0}^{\infty}
   \left|\frac{\partial^i g}{\partial t^i}(t,s)\right|
   \|\mathbf{C}\|_1 ds 
   \\
   &\leq&
   \|z_1-z_2\|_0
   \Bigg\{\mathcal{L}\Big(\|\Uplambda_1\|_1\Big)
   +\max\Big\{2\upeta,3\upeta^2\Big\}
    \mathcal{L}\Big(\|\Uplambda_2\|_1\Big)
    \\
   &&
   \hspace{2cm}
   +\max\Big\{2\upeta,3\upeta^2,4\upeta^3\Big\}
   \mathcal{L}\Big(\|\mathbf{C}\|_1\Big)\Bigg\}.
\end{eqnarray*}
Then, by application of $(\mathcal{R}_3)$, we deduce that $T$ is a contraction,
since, for an arbitrary $\upeta\in ]0,1/2[$, we have that $\max\Big\{ 2\upeta, 3\upeta^2 \Big\}
=\max\Big\{ 2\upeta, 3\upeta^2, 4\upeta^3\Big \}=2\upeta<1$.
Hence, from (a)-(c) and application of Banach fixed point theorem,
we deduce that there exists a unique $z\in D_\upeta\subset
C_0^2([t_0,\infty[)$ solution of \eqref{eq:operator_equation}.

\subsection{Proof of Theorem~\ref{lem:solution_perturbation}}

The proof of the Theorem~\ref{lem:solution_perturbation} follows by application 
Theorem~\ref{teo:solution_riccati_type_general}.
Indeed, in the next lines we verify the hypothesis
($\mathcal{R}_1$)-($\mathcal{R}_3$). First,
the hypothesis ($\mathcal{R}_1$) is satisfied 
since \eqref{eq:riccati_type} can 
be rewritten as \eqref{eq:riccati_type_general}. More precisely,
if $\lambda_i$ denotes an arbitrary
characteristic root of \eqref{eq:intro_dos},
we have that the constant coefficients $b_j, j=0,1,2,$ 
in \eqref{eq:riccati_type_general} are defined by
\begin{subequations}
\label{eq:notation_ricc}
\begin{eqnarray}
\begin{array}{rclcrcl}
  b_0 =4\lambda_i^3+3\lambda_i^2a_3+2\lambda_i a_2+a_1, 
  \quad
  b_1 =6\lambda_i^2+3\lambda_i a_3+a_2,
  \quad
  b_2 = 4\lambda_i+a_3,  
\end{array}
\label{eq:notation_ricc:1}
\end{eqnarray}
the functions $\Omega:\R\to\R$ and $\Lambda_1,\Lambda_2:\R\to\R^3$ 
and the constant $\mathbf{C}\in\R^7$ defining the function $F$ are  given by 
\begin{eqnarray}
\Omega(t) &=& -(\lambda_i^3r_3(t)+\lambda_i^2r_2(t)+\lambda_i r_1(t)+r_0(t)),
\qquad
\Lambda_1(t)=(b(t),f(t),h(t)),
\label{eq:notation_ricc:2}\\
\Lambda_2(t)&=&(p(t),f(t),h(t)),
\qquad
\mathbf{C}=
 -\Big(3,\;12\lambda_i +3a_3,
 \;6\lambda_i ^2+3\lambda_i a_3+a_2,
 \;4,\;6,\;4\lambda_i,1\Big),\quad{}
 \label{eq:notation_ricc:3}
\end{eqnarray}
with
\begin{eqnarray}
b(t) = -(3\lambda_i ^2r_3(t)+2\lambda_i r_2(t)+r_1(t)),
\;
f(t)=-(3\lambda_i r_3(t)+r_2(t)),
\;
p(t)=3h(t)=-3r_3(t).
\label{eq:notation_ricc:4}
\end{eqnarray}
\end{subequations}
In second place, by application of 
Proposition~\ref{lem:solucion_partelineal}, we deduce that 
the hypothesis ($\mathcal{R}_2$)  
is satisfied. Meanwhile, we note that (H$_2$)
implies ($\mathcal{R}_3$). 
Thus, we deduce that conclusion of the Theorem~\ref{lem:solution_perturbation}
is valid.

\subsection{Proof of Theorem~\ref{teo:solution_asymptotic}}
First we present some useful bounds
concerning to the Green functions $g_i$ defined on \eqref{eq:green_function}. 
In the case of $g_1$ and $g_4$, for $i\in\mathbb{N}\cup\{0\}$, 
we have the following bound
\begin{eqnarray}
&&\left|\frac{\partial^i g_\ell}{\partial t^i}(t,s)\right|
\leq
 \Big(|\gamma_3-\gamma_2||\gamma_1|^i
	+|\gamma_1-\gamma_3||\gamma_2|^i
	+|\gamma_2-\gamma_2||\gamma_3|^i\Big)
e^{-\alpha_\ell(t-s)},
\quad \ell=\{1,4\},
\quad{}
\label{eq18}
\end{eqnarray}
with $\alpha_1 =\max\Big\{\gamma_1,\gamma_2,\gamma_3\Big\}$ and
$\alpha_4 =\min\Big\{\gamma_1,\gamma_2,\gamma_3\Big\}$. 
Similarly for $g_2$ and $g_3$, for $i\in\mathbb{N}\cup\{0\}$, we  have that
\begin{eqnarray}
&&\left|\frac{\partial^i g_2}{\partial t^i}(t,s)\right|
\leq
\left\{
\begin{array}{lcl}
\Big(|\gamma_2-\gamma_1||\gamma_3|^i
      +|\gamma_3-\gamma_1||\gamma_2|^i\Big)e^{-\max\{\gamma_2,\gamma_3\}(t-s)},
      & \quad & t\le s,
 \\
|\gamma_3-\gamma_2||\gamma_1|^ie^{-\gamma_1(t-s)},
      & &t\ge s,
\end{array}
\right.
\label{eq18:g2}
\\
&&\left|\frac{\partial^i g_3}{\partial t^i}(t,s)\right|
\leq
\left\{
\begin{array}{lcl}
|\gamma_2-\gamma_1||\gamma_3|^ie^{-\gamma_3(t-s)},
      & &t\le s,
\\
\Big(|\gamma_3-\gamma_2||\gamma_1|^i
      +|\gamma_3-\gamma_1||\gamma_2|^i\Big)e^{-\min\{\gamma_1,\gamma_2\}(t-s)},
      & \quad & t\ge s.
\end{array}
\right.
\label{eq18:g3}
\end{eqnarray}
The proof of the bounds \eqref{eq18}-\eqref{eq18:g3} are straightforward by
application of the algebraic properties of the exponential function.

\vspace{0.5cm}
\noindent{\it Proof of \eqref{eq:asymptotic_teo:form} with $\ell=1$.}
Let us denote by  $T$ the operator defined in \eqref{eq:operator_fix_point}
and by  $z_1$ the solution of  the equation~\eqref{eq:riccati_type}
associated with the characteristic root $\lambda_1$ of \eqref{eq:intro_dos}.
Now,  on $D_\upeta$ with $\upeta\in ]0,1/2[,$ we 
define the sequence $\omega_{n+1}=T\omega_{n}$ with $\omega_0=0$, we have
that $\omega_n\to z_1$ when $n\to\infty$. This fact is a consequence 
of the contraction property of $T$. We note that the hypothesis (H$_1$)
and Proposition~\eqref{lem:solucion_partelineal} implies that
all roots of the corresponding characteristic
polynomial for \eqref{eq:riccati_type_equiv_homo} with $b_i$
defined on \eqref{eq:notation_ricc:1} are negative,
since,
$0>\gamma_1=\lambda_2-\lambda_1>
\gamma_2=\lambda_3-\lambda_1>
\gamma_3=\lambda_4-\lambda_1$. 
Moreover, by \eqref{eq:notation_ricc:2}
we note that the following identity $\Omega(s)=p(\lambda_1,s)$
holds.
Then, the Green
 function $g$ defined on \eqref{eq:green_function} is given by 
$(\delta\gamma)^{-1}g_1$. Naturally the operator $T$ can be rewrited equivalently
as follows
\begin{eqnarray}
Tz(t)&=&\frac{1}{\delta\gamma}\int_{t}^{\infty} g_1(t,s)\Big[p(\lambda_1,s)
+F\Big(s,z(s),z^{(1)}(s),z^{(2)}(s)\Big)\Big]ds,
\quad\mbox{for $t\ge t_0$},
\label{eq:operator_fix_point_g1}
\end{eqnarray}
since $g_1(t,s)=0$ for $s\in [t_0,t]$.
Thus,
the proof of \eqref{eq:asymptotic_teo:form} with $\ell=1$
is reduced to prove that
\begin{eqnarray}
&&\exists \quad \Phi_n\in \R_+ \quad :\quad 
\sum_{j=0}^2|\omega_n^{(j)}(t)|\leq
\Phi_n\int_{t}^{\infty}e^{-\beta(t-\tau)}|p(\lambda_1,\tau)|d\tau,
\quad\text{ $\forall$ }t\geq t_0,
\label{eq17}
\\
&& \mbox{$\exists\quad \Phi\in\R_+\quad\;\, :\quad$  $\Phi_n\to \Phi$ 
when $n\to \infty$}.
\label{eq17:uniformly_bounded}
\end{eqnarray}
Indeed, we prove \eqref{eq17} by induction on $n$
and deduce that \eqref{eq17:uniformly_bounded} is a consequence of the construction
of the sequence $\{\Phi_n\}$. Firstly, we prove \eqref{eq17}. Note that
for $n=1$ the estimate \eqref{eq17} is satisfied 
with $\Phi_1=A_1$.
It can be proved inmediatly by the definition of the operator $T$
given on \eqref{eq:operator_fix_point_g1}, the property  $F(s,0,0,0)=0$,
the estimate \eqref{eq18} and the hypothesis
that $\beta\in [\lambda_2-\lambda_1,0[$, since
\begin{eqnarray*}
  &&\sum_{j=0}^2|\omega^{(j)}_{1}(t)| 
  = \sum_{j=0}^2|T^{(j)}\omega_0(t)|
	=\frac{1}{|\delta\gamma|}
	\sum_{j=0}^2\int_{t}^{\infty}
	\left|\frac{\partial^i g_1}{\partial t^i}(t,s)\right||p(\lambda_1,s)|ds
	\\
	&&\qquad\le
	\frac{1}{|\delta\gamma|}\sum_{j=0}^2
\Big(|\gamma_3-\gamma_2||\gamma_1|^j
	+|\gamma_1-\gamma_3||\gamma_2|^j
	+|\gamma_2-\gamma_1||\gamma_3|^j\Big)
	\int_{t}^{\infty}e^{-(\lambda_2-\lambda_1)(t-\tau)}|p(\lambda_1,\tau)|d\tau
	\\
	&&\qquad
	=
	A_1
	\int_{t}^{\infty}e^{-(\lambda_2-\lambda_1)(t-\tau)}|p(\lambda_1,\tau)|d\tau
	\le
	A_1
	\int_{t}^{\infty}e^{-\beta(t-\tau)}|p(\lambda_1,\tau)|d\tau.
\end{eqnarray*} 
Now, assuming that \eqref{eq17} is valid for  
$n=k$, we prove that \eqref{eq17} is also valid for $n=k+1$.
However, before of prove the estimate \eqref{eq17} for $n=k+1$, we note that 
by the hypothesis (H$_3$) (i.e. the perturbations are belong to $\mathcal{F}_1([t_0,\infty[)$), 
the notation \eqref{eq:notation_ricc} and the fact
that $\max\{\upeta,\upeta^2,\upeta^3\}=\upeta$
for $\upeta\in]0,1/2[$, we deduce the following estimates
\begin{subequations}
\label{eq:prior_estimates}
\begin{eqnarray}
&&
\Big|p(\lambda_1,s)+F\Big(s,\omega_k(s),\omega^{(1)}_k(s),\omega^{(2)}_k(s)\Big)\Big|
	\nonumber\\
&&\qquad \le
	|a(s)|
	+|b(s)||\omega_k(s)|
	+|f(s)||\omega_k^{(1)}(s)|+|h(s)||\omega_k^{(2)}(s)| 
	\nonumber\\
   & & 
	\hspace{1cm}
	+|p(s)||\omega_k^{(1)}(s)||\omega_k(s)|
	+|f(s)||\omega_k(s)|^{2}+|h(s)||\omega_k(s)|^{3} 
	\nonumber\\
   & &
   \hspace{1cm}
	+|C_1||\omega_k^{(1)}(s)|^2
	+|C_2||\omega_k(s)||\omega_k^{(1)}(s)|
	+|C_3||\omega_k|^2
	+|C_4||\omega_k(s)||\omega_k^{(2)}(s)|
	\nonumber\\
	& & 
	\hspace{1cm}
	+|C_5||\omega_k(s)|^2|\omega_k^{(2)}(s)|
	+|C_6||\omega_k(s)|^3
	+|C_7||\omega_k(s)|^4 
	\nonumber\\
   &&\qquad \le 
   |p(\lambda_1,s)|
	+\Bigg[|b(s)|
	+|f(s)|+|h(s)| 
	+|p(s)|\upeta
	+|f(s)|\upeta+|h(s)|\upeta^2
	\nonumber\\
   & &
   \hspace{1cm}
	+\Big(\sum_{i=1}^4|C_i|\Big)\upeta
	+\Big(\sum_{i=5}^6|C_i|\Big)\upeta^2  \Bigg]
	\max\Big\{|\omega_k(s)|,|\omega^{(1)}_k(s)|,|\omega^{(2)}_k(s)|\Big\}
\nonumber\\
   && \qquad \le
   |p(\lambda_1,s)|
	+\Bigg[\|(b,f,h)(s)\|_1 +\Bigg(\|(p,f,h)(s)\|_1
	+\|\mathbf{C}\|_1\Bigg)\upeta \Bigg]
	\Big\|\Big(\omega_k,\omega^{(1)}_k,\omega^{(2)}_k\Big)(s)\Big\|_1,
\nonumber\\
&&
\label{eq:efe_estimate}
\\
&&
\int_{t}^{\infty}e^{-(\lambda_2-\lambda_1)(t-s)}|b(s)|ds
\le (3|\lambda_1|^2+2|\lambda_1|+1)\rho_1,
\label{eq:efe_estimate_b}
\\
&&
\int_{t}^{\infty}e^{-(\lambda_2-\lambda_1)(t-s)}|f(s)|ds
\le (3|\lambda_1|+1)\rho_1,
\label{eq:efe_estimate_f}
\\
&&
\int_{t}^{\infty}e^{-(\lambda_2-\lambda_1)(t-s)}|p(s)|ds
\le 3\rho_1,
\qquad
\int_{t}^{\infty}e^{-(\lambda_2-\lambda_1)(t-s)}|h(s)|ds
\le \rho_1.
\label{eq:efe_estimate_ph}
\end{eqnarray}
\end{subequations}
Using \eqref{eq:operator_fix_point_g1},
the notation \eqref{eq:notation_ricc}, the inductive hypothesis,
the inequality \eqref{eq18} and the estimates \eqref{eq:prior_estimates}
we have that
\begin{eqnarray*}
 && \sum_{j=0}^2|\omega^{(j)}_{k+1}(t)| 
  = \sum_{j=0}^2|T^{(j)}\omega_k(t)|
  \\
	&&
	\hspace{0.5cm}
	=
	\frac{1}{|\delta\gamma|}\sum_{j=0}^2\left|\int_{t}^{\infty}
	\frac{\partial^i g_1}{\partial t^i}(t,s)
	\Big[p(\lambda_1,s)
	 +F\Big(s,\omega_k(s),\omega^{(1)}_k(s),\omega^{(2)}_k(s)\Big)\Big]ds\right|
\\
&&\hspace{0.5cm}
\le\frac{1}{|\delta\gamma|}\sum_{j=0}^2\int_{t}^{\infty}
	\left|\frac{\partial^i g_1}{\partial t^i}(t,s)\right|
	\Big|p(\lambda_1,s)
	 +F\Big(s,\omega_k(s),\omega^{(1)}_k(s),\omega^{(2)}_k(s)\Big)\Big|ds
\\
&&\hspace{0.5cm}
\le
A_1\int_{t}^{\infty}e^{-(\lambda_2-\lambda_1)(t-s)}
\Bigg\{|p(\lambda_1,s)|+\Bigg[\|(b,f,h)(s)\|_1 
\\
&&\hspace{1cm}+\Bigg(\|(p,f,h)(s)\|_1
	+\|\mathbf{C}\|_1\Bigg)\upeta \Bigg]
	\Big\|\Big(\omega_k,\omega^{(1)}_k,\omega^{(2)}_k\Big)(s)\Big\|_1\Bigg\}ds
\\
&&\hspace{0.5cm}
\le
A_1\int_{t}^{\infty}e^{-(\lambda_2-\lambda_1)(t-s)}
\Bigg\{|p(\lambda_1,s)|+\Bigg[\|(b,f,h)(s)\|_1 
\\
&&\hspace{1cm}+\Bigg(\|(p,f,h)(s)\|_1
	+\|\mathbf{C}\|_1\Bigg)\upeta \Bigg]
	\Phi_k\int_{s}^{\infty}e^{-\beta(s-\tau)}|p(\lambda_1,\tau)| d\tau\Bigg\}ds
\\
&&\hspace{0.5cm}
\le
A_1\Bigg\{1+
\int_{t}^{\infty}e^{-(\lambda_2-\lambda_1)(t-s)}
\Bigg[\|(b,f,h)(s)\|_1 +\Bigg(\|(p,f,h)(s)\|_1
	+\|\mathbf{C}\|_1\Bigg)\upeta \Bigg]
	\Phi_k ds\Bigg\}
\\
&&\hspace{1cm}\quad
\times
	\Bigg\{
	\int_{t}^{\infty}e^{-\beta(t-\tau)}|p(\lambda_1,\tau)| d\tau\Bigg\}
\\
&&\hspace{0.5cm}
\le 
   A_1\Big(
	1
	+\Phi_k\;\rho_1\;
	\upvarsigma_1\Big)
	\int_{t}^{\infty} e^{-\beta (t-\tau)}|p(\lambda_1,\tau)|d\tau,
\end{eqnarray*}
Then, by the induction process, \eqref{eq17} is satisfied with 
$\Phi_{n}=   A_1(1+\Phi_{n-1}\;\rho_1\;\upvarsigma_1).$
Now, using recursively the definition of $\Phi_{n-2},\ldots,\Phi_{2}$,
we can rewrite $\Phi_{n}$ as the sum of terms of 
a geometric progression where the common ratio is given
by $\rho_1A_1\upvarsigma_1$. Then, the existence
of $\Phi$ satisfying \eqref{eq17:uniformly_bounded} follows
by the hypothesis that
$\rho_1A_1\upvarsigma_1\in ]0,1[$
for $\upeta\in ]0,1/2[.$
More precisely, we have that         
\begin{eqnarray*}
\lim_{n\to\infty}
\Phi_{n}=A_1\lim_{n\to\infty}\sum_{i=0}^{n-1}
	\Big(\rho_1A_1\upvarsigma_1\Big)^i
        =A_1
        \lim_{n\to\infty}
        \frac{\Big[(\rho_1A_1\upvarsigma_1)^n-1\Big]}
        {\rho_1A_1\upvarsigma_1-1}
        =
        \frac{A_1}{1-\rho_1A_1\upvarsigma_1}
        =\Phi>0.
\end{eqnarray*}
Hence, \eqref{eq17}-\eqref{eq17:uniformly_bounded} are  valid and
the proof of \eqref{eq:asymptotic_teo:form} with $\ell=1$
is concluded by passing to the limit the sequence $\{\Phi_{n}\}$ 
when $n\to\infty$ and
in the topology of $C^2_0([t_0,\infty]).$

\vspace{0.5cm}
\noindent{\it Proof of \eqref{eq:asymptotic_teo:form} with $\ell=2$.}
Let us denote by $z_2$ the solution of the equation~\eqref{eq:riccati_type} 
associated with the characteristic root $\lambda_2$ of \eqref{eq:intro_dos}.
Similarly to the case $\ell=1$ we 
define the sequence $\omega_{n+1}=T\omega_{n}$ with $\omega_0=0$ and,
by the contraction property of $T$, we can deduce that
$\omega_n\to z_2$ when $n\to\infty$. 
In this case, we note that $\Omega(s)=p(\lambda_2,s)$.
Moreover, by Proposition~\eqref{lem:solucion_partelineal} we have that
$\gamma_1=\lambda_1-\lambda_2>0>
\gamma_2=\lambda_3-\lambda_2>
\gamma_3=\lambda_4-\lambda_2$. Then the Green
 function $g$ defined on \eqref{eq:green_function} is given by 
$(\delta\gamma)^{-1}g_2$.
Thereby,  the operator $T$ can be rewrited equivalently
as follows
\begin{eqnarray}
Tz(t)&=&\frac{1}{\delta\gamma}\int_{t_0}^{\infty} g_2(t,s)
\Big[p(\lambda_2,s)
+F\Big(s,z(s),z^{(1)}(s),z^{(2)}(s)\Big)\Big]ds
\nonumber\\
&=&
\frac{1}{\delta\gamma} 
\Bigg\{
\int_{t_0}^{t} 
	(\gamma_2-\gamma_3)e^{-\gamma_1(t-s)}
	\Big[p(\lambda_2,s)
	+F\Big(s,z(s),z^{(1)}(s),z^{(2)}(s)\Big)\Big]ds
\nonumber\\
&&
+\int_{t}^{\infty} 
	\Big[(\gamma_2-\gamma_1)e^{-\gamma_3(t-s)}
	-(\gamma_3-\gamma_1)e^{-\gamma_2(t-s)}\Big]
\nonumber\\
&&\hspace{2cm}\times
	\Big[p(\lambda_2,s)
+F\Big(s,z(s),z^{(1)}(s),z^{(2)}(s)\Big)\Big]ds
\Bigg\}.
\label{eq:operator_fix_point_g2}
\end{eqnarray}
Then, the proof of \eqref{eq:asymptotic_teo:form} with $\ell=2$
is reduced to prove 
\begin{eqnarray}
&&\exists \quad \Phi_n\in \R_+ \quad :\quad 
\sum_{j=0}^2|\omega_n^{(j)}(t)|\leq
\Phi_n\int_{t_0}^{\infty}e^{-\beta(t-\tau)}|p(\lambda_2,\tau)|d\tau,\quad\text{ $\forall$ }t\geq t_0,
\label{eq17_2}
\\
&& \mbox{$\exists\quad \Phi\in\R_+\quad\;\, :\quad$  $\Phi_n\to \Phi$ 
when $n\to \infty$}.
\label{eq17:uniformly_bounded_2}
\end{eqnarray}
In the induction step for $n=1$ the estimate \eqref{eq17_2} is satisfied 
with $\Phi_1=A_2$, 
since by the definition of the operator $T$
given on \eqref{eq:operator_fix_point_g2}, the property  $F(s,0,0,0)=0$,
the estimates of type \eqref{eq18:g2}   and the fact
that $\beta\in[\lambda_3-\lambda_2,0[\subset[\lambda_3-\lambda_2,\lambda_1-\lambda_2]$, 
we deduce the following bound
\begin{eqnarray*}
  \sum_{j=0}^2|\omega^{(j)}_{1}(t)| 
  &=& \sum_{j=0}^2|T^{(j)}\omega_0(t)|
	\\
	&\le&
	\frac{1}{|\delta\gamma|}
	\Bigg(
	\sum_{j=0}^2
	|\gamma_2-\gamma_3||\gamma_1|^j
	\Bigg)
	\int_{t_0}^{t}e^{-(\lambda_1-\lambda_2)(t-s)}|p(\lambda_2,s)|ds
	\\
	&&
	\quad
	+\frac{1}{|\delta\gamma|}
	\Bigg(
	\sum_{j=0}^2|\gamma_2-\gamma_1||\gamma_3|^j+
	|\gamma_3-\gamma_1||\gamma_2|^j
	\Bigg)
	\int_{t}^{\infty}e^{-(\lambda_3-\lambda_2)(t-s)}|p(\lambda_2,s)|ds
	\\
	&\le&
	A_2
	\left\{
	\int_{t_0}^{t}e^{-\beta(t-s)}|p(\lambda_2,\tau)|ds
	+
	\int_{t}^{\infty}e^{-\beta(t-s)}|p(\lambda_2,\tau)|ds
	\right\}
	\\
	&=&
	A_2\int_{t_0}^{\infty}e^{-\beta(t-\tau)}|p(\lambda_2,\tau)|d\tau.
\end{eqnarray*}
Noticing that a similar inequalities to \eqref{eq:prior_estimates},
with $\lambda_2$ instead of $\lambda_1$ and integration on
$[t_0,\infty[$ instead of $[t,\infty[$, we deduce that 
\begin{eqnarray*}
	J_1&:=&\int_{t_0}^{t}e^{-(\lambda_1-\lambda_2)(t-s)}|p(\lambda_2,s)|ds
	+
	\int_{t}^{\infty}e^{-(\lambda_3-\lambda_2)(t-s)}|p(\lambda_2,s)|ds
	\\
	&\le&
	\int_{t_0}^{t}e^{-\beta(t-s)}|p(\lambda_2,\tau)|ds
	+
	\int_{t}^{\infty}e^{-\beta(t-s)}|p(\lambda_2,\tau)|ds
	=
	\int_{t_0}^{\infty}e^{-\beta(t-\tau)}|p(\lambda_2,\tau)|d\tau,
\\
	J_2&:=&
	\int_{t_0}^t
	e^{-(\lambda_1-\lambda_2)(t-s)}
	\Bigg[\|(b,f,h)(s)\|_1 
	+\Bigg(\|(p,f,h)(s)\|_1
	+\|\mathbf{C}\|_1\Bigg)\upeta \Bigg]
\\
	&&
	\hspace{1cm}
	\times
	\Big\|\Big(\omega_k,\omega^{(1)}_k,\omega^{(2)}_k\Big)(s)\Big\|_1\Bigg\}ds
	+\int_{t}^{\infty}
	e^{-(\lambda_1-\lambda_2)(t-s)}
	\Bigg[\|(b,f,h)(s)\|_1
\\
	&&
	\hspace{1cm}	
	+\Bigg(\|(p,f,h)(s)\|_1
	+\|\mathbf{C}\|_1\Bigg)\upeta \Bigg]
	\times
	\Big\|\Big(\omega_k,\omega^{(1)}_k,\omega^{(2)}_k\Big)(s)\Big\|_1\Bigg\}ds
\\
	&\le&
	\int_{t_0}^t
	e^{-\beta(t-s)}
	\Bigg[\|(b,f,h)(s)\|_1 
	+\Bigg(\|(p,f,h)(s)\|_1
	+\|\mathbf{C}\|_1\Bigg)\upeta \Bigg]
	\Big\|\Big(\omega_k,\omega^{(1)}_k,\omega^{(2)}_k\Big)(s)\Big\|_1\Bigg\}ds
\\
	&&
	\;
	+\int_{t}^{\infty}
	e^{-\beta(t-s)}
	\Bigg[\|(b,f,h)(s)\|_1 
	+\Bigg(\|(p,f,h)(s)\|_1
	+\|\mathbf{C}\|_1\Bigg)\upeta \Bigg]
	\Big\|\Big(\omega_k,\omega^{(1)}_k,\omega^{(2)}_k\Big)(s)\Big\|_1\Bigg\}ds
\\
	&=&
	\int_{t_0}^{\infty}
	e^{-\beta(t-s)}
	\Bigg[\|(b,f,h)(s)\|_1 
	+\Bigg(\|(p,f,h)(s)\|_1
	+\|\mathbf{C}\|_1\Bigg)\upeta \Bigg]
	\Big\|\Big(\omega_k,\omega^{(1)}_k,\omega^{(2)}_k\Big)(s)\Big\|_1\Bigg\}ds
\\
	&\le&
	\Phi_k
	\int_{t_0}^{\infty}
	e^{-\beta(t-s)}
	\Bigg[\|(b,f,h)(s)\|_1 
	+\Bigg(\|(p,f,h)(s)\|_1
	+\|\mathbf{C}\|_1\Bigg)\upeta \Bigg]
	\int_{t_0}^{\infty}e^{-\beta(s-\tau)}|p(\lambda_2,\tau)|d\tau
	\Bigg\}ds
\\
	&=&
	\Phi_k
	\int_{t_0}^{\infty}
	\Bigg[\|(b,f,h)(s)\|_1 
	+\Bigg(\|(p,f,h)(s)\|_1
	+\|\mathbf{C}\|_1\Bigg)\upeta \Bigg]
	\int_{t_0}^{\infty}e^{-\beta(t-\tau)}|p(\lambda_2,\tau)|d\tau
	\Bigg\}ds
\\
	&\le&
	\Phi_k\;\rho_2\;\upvarsigma_2
	\int_{t_0}^{\infty}e^{-\beta(t-\tau)}|p(\lambda_2,\tau)|d\tau.
\end{eqnarray*}
Then, the general induction step can be proved as follows
\begin{eqnarray*}
&&  \sum_{j=0}^2|\omega^{(j)}_{k+1}(t)| 
  = \sum_{j=0}^2|T^{(j)}\omega_k(t)| 
	\\
&&\qquad\le
	\frac{1}{|\delta\gamma|}
	\Bigg(
	\sum_{j=0}^2
	|\gamma_2-\gamma_3||\gamma_1|^j
	\Bigg)
	\int_{t_0}^{t}e^{-(\lambda_1-\lambda_2)(t-s)}
	\Big|p(\lambda_2,s)+F\Big(s,z(s),z^{(1)}(s),z^{(2)}(s)\Big)\Big|ds
	\\
&&\qquad\qquad
  	+\frac{1}{|\delta\gamma|}
	\Bigg(
	\sum_{j=0}^2|\gamma_1+\gamma_2||\gamma_3|^j+
	|\gamma_1+\gamma_3||\gamma_2|^j
	\Bigg)
	\\
&&\qquad\qquad\qquad
	\times
	\int_{t}^{\infty}e^{-(\lambda_3-\lambda_2)(t-s)}
	\Big|p(\lambda_2,s)+F\Big(s,z(s),z^{(1)}(s),z^{(2)}(s)\Big)\Big|ds
	\\
&&\qquad\le
	A_2
	\Bigg[
	\int_{t_0}^t
	e^{-(\lambda_1-\lambda_2)(t-s)}
	\Bigg\{|p(\lambda_2,s)|+\Bigg[\|(b,f,h)(s)\|_1 
	+\Bigg(\|(p,f,h)(s)\|_1
	+\|\mathbf{C}\|_1\Bigg)\upeta \Bigg]
\\
&&\qquad\qquad
	\times
	\Big\|\Big(\omega_k,\omega^{(1)}_k,\omega^{(2)}_k\Big)(s)\Big\|_1\Bigg\}ds
	+\int_{t}^{\infty}
	e^{-(\lambda_3-\lambda_2)(t-s)}
	\Bigg\{|p(\lambda_2,s)|+\Bigg[\|(b,f,h)(s)\|_1 
\\
&&\qquad\qquad
	+\Bigg(\|(p,f,h)(s)\|_1
	+\|\mathbf{C}\|_1\Bigg)\upeta \Bigg]
	\Big\|\Big(\omega_k,\omega^{(1)}_k,\omega^{(2)}_k\Big)(s)\Big\|_1\Bigg\}ds
	\Bigg]
\\&&\qquad=
	A_2 \Big[J_1+J_2\Big]
\\
&&\qquad
\le 
   A_2\Big(
	1
	+\Phi_k\;\rho_2\;
	\upvarsigma_2\Big)
	\int_{t_0}^{\infty} e^{-\beta (t-\tau)}|p(\lambda_2,\tau)|d\tau.
\end{eqnarray*} 
Hence the thesis of the inductive steps holds with 
$\Phi_{n}=A_2(1+\Phi_{n-1}\rho_2\;\upvarsigma_2)$. 
Now, proceeding in analogous way to the case  $\ell=1$, 
we find that \eqref{eq17:uniformly_bounded_2} is satisfied
with $\Phi=A_2/(1-\rho_2\;\upvarsigma_2A_2)>0$.
Thus, the sequence $\{\Phi_{n}\}$ is convergent 
and $z_2$ (the limit of $\omega_n$
in the topology of $C^2_0([t_0,\infty])$) satisfies \eqref{eq:asymptotic_teo:form}.

\vspace{0.5cm}
\noindent{\it Proof of \eqref{eq:asymptotic_teo:form} with $\ell=3$ .}
The proof of the case $\ell=3$ is similar to case $\ell=2$.

\vspace{0.5cm}
\noindent{\it Proof of \eqref{eq:asymptotic_teo:form} with $\ell=4$ .}
Similarly to the preceding cases,
let us denote by $z_4$ the solution of the equation~\eqref{eq:riccati_type} 
associated with the characteristic root $\lambda_4$ of \eqref{eq:intro_dos}.
We start be defining the sequence $\omega_{n+1}=T\omega_{n}$ with $\omega_0=0$ and
and note that by the contraction property of $T$, we can deduce that
$\omega_n\to z_4$ when $n\to\infty$. 
Now, in this case, we have that $\Omega(s)=p(\lambda_4,s)$,
$\gamma_1=\lambda_1-\lambda_4>
\gamma_2=\lambda_2-\lambda_4>
\gamma_3=\lambda_3-\lambda_4>0$ (see Proposition~\eqref{lem:solucion_partelineal})
and $g=(\delta\gamma)^{-1}g_4$  (see \eqref{eq:green_function}).
Then, we can deduce that the operator $T$ is given by
\begin{eqnarray}
Tz(t)&=&\frac{1}{\delta\gamma}\int_{t_0}^{\infty} g_4(t,s)
\Big[p(\lambda_2,s)
+F\Big(s,z(s),z^{(1)}(s),z^{(2)}(s)\Big)\Big]ds
\nonumber\\
&=&
\frac{1}{\delta\gamma} 
\Bigg\{
\int_{t_0}^{t} 
	\Big[(\gamma_3-\gamma_2)e^{-\gamma_1(t-s)}
	+(\gamma_1-\gamma_3)e^{-\gamma_2(t-s)}
	+(\gamma_2-\gamma_1)e^{-\gamma_3(t-s)}
	\Big]
\nonumber\\
&&\hspace{1.5cm}\times
	\Big[p(\lambda_2,s)
+F\Big(s,z(s),z^{(1)}(s),z^{(2)}(s)\Big)\Big]ds
\Bigg\},
\quad \mbox{for $t\ge t_0$.}
\label{eq:operator_fix_point_g4}
\end{eqnarray}
Thus, to proof \eqref{eq:asymptotic_teo:form} with $\ell=4$
is enough prove the following facts
\begin{eqnarray}
&&\exists \quad \Phi_n\in \R_+ \quad :\quad 
\sum_{j=0}^2|\omega_n^{(j)}(t)|\leq
\Phi_n\int_{t_0}^{t}e^{-\beta(t-\tau)}|p(\lambda_4,\tau)|d\tau,\quad\text{ $\forall$ }t\geq t_0,
\label{eq17_4}
\\
&& \mbox{$\exists\quad \Phi\in\R_+\quad\;\, :\quad$  $\Phi_n\to \Phi$ 
when $n\to \infty$}.
\label{eq17:uniformly_bounded_4}
\end{eqnarray}
Now, by \eqref{eq:operator_fix_point_g4}, \eqref{eq18} and the
roperty  $F(s,0,0,0)=0$,
we note the induction step for $n=1$ the estimate \eqref{eq17_4} is satisfied 
with $\Phi_1=A_4$ and  $\beta\in]0,\lambda_3-\lambda_4]$. Indeed,
we can deduce the following estimate
\begin{eqnarray*}
  &&\sum_{j=0}^2|\omega^{(j)}_{1}(t)| 
      = \sum_{j=0}^2|T^{(j)}\omega_0(t)|
	\\
	&&
	\qquad
	\le
	\frac{1}{|\delta\gamma|}
	\sum_{j=0}^2
	|\gamma_3-\gamma_2||\gamma_1|^j
	+|\gamma_1-\gamma_3||\gamma_2|^j+
	|\gamma_2-\gamma_1||\gamma_3|^j
	\int_{t_0}^{t}e^{-(\lambda_3-\lambda_4)(t-s)}|p(\lambda_4,s)|ds
	\\
	&&\qquad\le
	A_4\int_{t_0}^{t}e^{-\beta(t-\tau)}|p(\lambda_4,\tau)|d\tau.
\end{eqnarray*}
Then,  we can prove 
that the general induction step holds with 
$\Phi_{n}=A_4(1+\Phi_{n-1}\rho_4\;\upvarsigma_4)$,
since proceeding as in the case $\ell=1$, we can
deduce the following estimate
\begin{eqnarray*}
&&  \sum_{j=0}^2|\omega^{(j)}_{k+1}(t)| 
  = \sum_{j=0}^2|T^{(j)}\omega_k(t)| 
\le 
   A_4\Big(
	1
	+\Phi_k\;\rho_4\;
	\upvarsigma_4\Big)
	\int_{t_0}^{t} e^{-\beta (t-\tau)}|p(\lambda_4,\tau)|d\tau.
\end{eqnarray*} 
Hence \eqref{eq17:uniformly_bounded_4} is satisfied
with $\Phi=A_4/(1-\rho_4\;\upvarsigma_4A_4)>0$.
Thus, the sequence $\{\Phi_{n}\}$ is convergent 
and $z_4$ (the limit of $\omega_n$
in the topology of $C^2_0([t_0,\infty])$) satisfies \eqref{eq:asymptotic_teo:form}.

\subsection{Proof of Theorem~\ref{teo1.7.1}}
By Lemma~\ref{lem:cambio_var}, we have that the fundamental system of solutions for
\eqref{eq:intro_uno} is given by \eqref{eq:fundam_sist_sol}.
Moreover, by \eqref{eq:fundam_sist_sol} we deduce the identities
\begin{eqnarray}
\frac{y_i^{(1)}(t)}{y_i(t)}
	 &=&[\lambda_i+z_i(t)]
\label{eq:deri_uno}\\
\frac{y_i^{(2)}(t)}{y_i(t)} 
	&=& [\lambda_i+z_i(t)]^2+z_i^{(1)}(t), 
\label{eq:deri_dos}\\
\frac{y_i^{(3)}(t)}{y_i(t)} 
	&=&
[\lambda_i+z_i(t)]^3+3[\lambda_i+z_i(t)]z_i^{(1)}(t)+z_i^{(2)}(t),
\label{eq:deri_tres}\\
\frac{y_i^{(4)}(t)}{y_i(t)} 
	&=&
[\lambda_i+z_i(t)]^4
	+6[\lambda_i+z_i(t)]^2z_i^{(1)}(t)
	+3[z^{(1)}_i(t)]^2
\nonumber\\
	&&
	+4[\lambda_i+z_i(t)]z_i^{(2)}(t)
	+z_i^{(2)}(t),
\label{eq:deri_cuatro}
\end{eqnarray}
Now, using the facts that $z_i\in C_0^2([t_0,\infty[)$
and  $\{\mu_i,z_i\}$ is a solution of \eqref{eq:riccati_type}, we deduce
the proof of \eqref{eq:asymptotic_perturbed}.
Now, by the definition of the $W[y_1,\ldots,y_n]$,
some algebraic rearrangements and 
\eqref{eq:asymptotic_perturbed}, we deduce \eqref{eq:asymptotic_perturbed_w}.

The proof of \eqref{eq:asymptotic_perturbed_2} follows by 
the identity
\begin{eqnarray}
\int_{t_0}^t e^{-a\tau}\int_{\tau}^\infty e^{-as} H(s)dsd\tau
&=&-\frac{1}{a}\left[
\int_{t}^\infty e^{-a(t-s)}H(s)ds-\int_{t_0}^\infty e^{-a(t_0-s)}H(s)ds
\right]
\nonumber\\
&&
\qquad
+\frac{1}{a}\int_{t_0}^t H(\tau)d\tau
\label{eq:idenaux_pro}
\end{eqnarray}
and by \eqref{eq:operator_fix_point}-\eqref{eq:operator_equation}. 
Now we develop the proof for $i=1$.
Indeed,
by \eqref{eq:fundam_sist_sol} we have that
\begin{eqnarray}
y_1(t)=\exp\Big(\int_{t_0}^t(\lambda_1+z_1(\tau))d\tau\Big)
=e^{\lambda_1(t-t_0)}\exp\Big(\int_{t_0}^t z_1(\tau)d\tau\Big).
\label{eq:solfun_paso1}
\end{eqnarray}
By \eqref{eq:operator_fix_point}-\eqref{eq:operator_equation}
and \eqref{eq:idenaux_pro}, we have that
\begin{eqnarray*}
\int_{t_0}^t z_1(\tau)d\tau
&=& \frac{1}{\delta\gamma}\int_{t_0}^t \int_{t_0}^\infty g_1(\tau,s)
	\Big(p(\lambda_1,s)+F(s,z_1(s),z_1^{(1)}(s),z_1^{(2)}(s)\Big)d\tau
\\
&=& \frac{1}{\delta\gamma}\int_{t_0}^t \int_{\tau}^\infty 
	\Big[
	(\gamma_3-\gamma_2)e^{-\gamma_1(\tau-s)}
	+(\gamma_1-\gamma_3)e^{-\gamma_2(\tau-s)}
	+(\gamma_2-\gamma_1)e^{-\gamma_3(\tau-s)}
	\Big]
\\
&&\hspace{1.8cm}\times
	\Big(p(\lambda_1,s)+F(s,z_1(s),z_1^{(1)}(s),z_1^{(2)}(s))\Big)d\tau
\\
&=&
\frac{1}{\delta\gamma}
	\left[
	\frac{\gamma_3-\gamma_1}{\gamma_1}
	+\frac{\gamma_1-\gamma_3}{\gamma_2}
	+\frac{\gamma_2-\gamma_1}{\gamma_3} 
	\right]
	\int_{t_0}^t\Big(p(\lambda_1,s)+F(s,z_1(s),z_1^{(1)}(s),z_1^{(2)}(s))\Big)d\tau
\\
&&
+\frac{1}{\delta\gamma}
	\Bigg[\frac{\gamma_3-\gamma_1}{\gamma_1}
	\Bigg\{
	\int_{t}^\infty 
	e^{-\gamma_1(t-s)}
	\Big(p(\lambda_1,s)+F(s,z_1(s),z_1^{(1)}(s),z_1^{(2)}(s))\Big)ds
\\
&&
\hspace{2.5cm}
	-\int_{t_0}^\infty 
	e^{-\gamma_1(t_0-s)}
	\Big(p(\lambda_1,s)+F(s,z_1(s),z_1^{(1)}(s),z_1^{(2)}(s))\Big)ds
	\Bigg\}
	\Bigg]
\\
&&
+\frac{1}{\delta\gamma}
	\Bigg[\frac{\gamma_1-\gamma_3}{\gamma_2}
	\Bigg\{
	\int_{t}^\infty 
	e^{-\gamma_2(t-s)}
	\Big(p(\lambda_1,s)+F(s,z_1(s),z_1^{(1)}(s),z_1^{(2)}(s))\Big)ds
\\
&&
\hspace{2.5cm}
	-\int_{t_0}^\infty 
	e^{-\gamma_2(t_0-s)}
	\Big(p(\lambda_1,s)+F(s,z_1(s),z_1^{(1)}(s),z_1^{(2)}(s))\Big)ds
	\Bigg\}
	\Bigg]
\\
&&
+\frac{1}{\delta\gamma}
	\Bigg[\frac{\gamma_2-\gamma_1}{\gamma_3}
	\Bigg\{
	\int_{t}^\infty 
	e^{-\gamma_3(t-s)}
	\Big(p(\lambda_1,s)+F(s,z_1(s),z_1^{(1)}(s),z_1^{(2)}(s))\Big)ds
\\
&&
\hspace{2.5cm}
	-\int_{t_0}^\infty 
	e^{-\gamma_3(t_0-s)}
	\Big(p(\lambda_1,s)+F(s,z_1(s),z_1^{(1)}(s),z_1^{(2)}(s))\Big)ds
	\Bigg\}
	\Bigg]
\\
&=&
\frac{1}{\gamma_1\gamma_2\gamma_3}
	\int_{t_0}^t\Big(p(\lambda_1,s)+F(s,z_1(s),z_1^{(1)}(s),z_1^{(2)}(s))\Big)d\tau
	+o(1)
\end{eqnarray*}
Then, \eqref{eq:asymptotic_perturbed_2} is valid for $i=1$, since
$\gamma_1\gamma_2\gamma_3=(\lambda_2-\lambda_1)(\lambda_3-\lambda_1)(\lambda_4-\lambda_1)=\pi_1.$
The proof of \eqref{eq:asymptotic_perturbed_2} for $i=2,3,4$ is analogous.
Now the proof of  \eqref{eq:asymptotic_perturbed_3} follows by
\eqref{eq:asymptotic_perturbed_2} and \eqref{eq:deri_uno}-\eqref{eq:deri_cuatro}.

\section*{Acknowledgement}

An{\'\i}bal Coronel and Fernando Huancas
thanks for the  support of Fondecyt project 11060400; and
the research projects 124109 3/R, 104709 01 F/E and 
121909 GI/C at Universidad del B{\'\i}o-B{\'\i}o, Chile.
Manuel Pinto thanks for the  support of Fondecyt project
1120709.

\end{document}